\documentclass{amsart}
\usepackage{amsfonts}
\usepackage{amssymb}
\numberwithin{equation}{section}
\theoremstyle{plain}
 \newtheorem{thm}{Theorem}[section]
 \newtheorem{lem}[thm]{Lemma}
 
 \newtheorem{prop}[thm]{Proposition}
\theoremstyle{definition}

 \newtheorem{rem}[thm]{Remark}
\newcommand{\al}{\alpha}%\allow

\newcommand{\gm}{\gamma}

\newcommand{\dl}{\delta}

\newcommand{\ld}{\lambda}%\ldots

\newcommand{\sg}{\sigma}

\newcommand{\rh}{\rho}

\newcommand{\dar}{\downarrow}

\newcommand{\mbf}{\mathbf}
\newcommand{\mcal}{\mathcal}
\newcommand{\mrm}{\mathrm}

\newcommand{\wh}{\widehat}

\newcommand{\R}{\mathbb{R}}
\newcommand{\law}{\mathcal L}

\begin{document}
\title{Selfdecomposability and semi-selfdecomposability
in subordination of cone-parameter convolution semigroups}
\author{Ken-iti Sato}
\address{Hachiman-yama 1101-5-103, Tenpaku-ku\\ Nagoya, 468-0074 Japan} 
\email{ken-iti.sato@nifty.ne.jp\;}
\begin{abstract}
Extension of two known facts concerning subordination is made. The first fact is that, 
in subordination of
$1$-dimensional Brownian motion with drift,
selfdecomposability is inherited from subordinator to subordinated.
This is extended to subordination of cone-parameter convolution semigroups.
The second fact is that,
in subordination of strictly stable cone-parameter convolution semigroups
on $\mathbb{R}^d$, 
selfdecomposability is inherited from subordinator to subordinated.
This is extended to semi-selfdecomposability.
\end{abstract}
\maketitle

\section{Introduction}

A subset $K$ of $\R^N$ is called a cone
 if it is a non-empty closed convex set which is
closed under multiplication by nonnegative reals and contains no straight line through
$0$ and if $K\ne\{0\}$.  
Given a cone $K$, we call  $\{\mu_s\colon s\in K\}$ a $K$-parameter convolution
  semigroup on $\R^d$ if it is a family of probability measures  on $\R^d$ satisfying  
\begin{gather}
\mu_{s_1}*\mu_{s_2}=\mu_{s_1+s_2}\quad\text{for }s_1, s_2\in K,\label{1.1}\\
\mu_{t s}\rightarrow \dl_0\quad\text{as $t\dar0$, for $s\in K$},\label{1.2}
\end{gather}
where $\dl_0$ is delta distribution located at $0\in\R^d$.
Convergence of probability measures is understood as weak convergence.
It follows from \eqref{1.1} and \eqref{1.2} that $\mu_0=\dl_0$.

Subordination of a cone-parameter convolution semigroup is defined as follows.
Let $K_1$ and $K_2$ be cones in $\R^{N_1}$ and $\R^{N_2}$, respectively.
Let $\{\mu_u\colon u\in K_2\}$ be a $K_2$-parameter
convolution semigroup on $\R^d$ and $\{\rh_s\colon s\in K_1\}$
a   $K_1$-parameter convolution semigroup on $\R^{N_2}$ supported on $K_2$ 
(that is, $\mrm{Supp}(\rh_s)\subseteq K_2$).
 Define a probability
measure
$\sg_s$ on $\R^d$ by
\begin{equation}\label{1.3}
\sg_s(B)=\int_{K_2} \mu_u(B) \rh_s(du)\quad\text{for }B\in\mcal B(\R^d),
\end{equation}
where $\mcal B(\R^d)$ is the class of Borel sets in $\R^d$.
Then $\{\sg_s\colon s\in K_1\}$
is
a $K_1$-parameter convolution semigroup on $\R^d$.
This procedure to get $\{\sg_s\colon s\in K_1\}$ is called subordination of
 $\{\mu_u\colon u\in K_2\}$ by  $\{\rh_s\colon s\in K_1\}$.
Convolution semigroups
$\{\mu_u\colon u\in K_2\}$, $\{\rh_s\colon s\in K_1\}$, and  $\{\sg_s\colon s\in
K_1\}$
are respectively called subordinand, subordinating (or subordinator), and 
subordinated.

Cone-parameter convolution semigroups on $\R^d$ and their subordination are
introduced in Pedersen and Sato \cite{PS03}.  Their basic properties are proved 
in Theorems 2.8, 2.11, and 4.4 of
\cite{PS03}. A number of examples are given there.
In Barndorff-Nielsen, Pedersen, and Sato \cite{BPS01}, 
several models leading to
$\R_+$-parameter convolution semigroups supported on $\R_+^N$ are discussed, 
including some financial models.
Here $\R_+=[0,\infty)$ and $\R_+^N=(\R_+)^N$.

In $\R_+$-parameter case, any 
convolution semigroup on $\R^d$ corresponds to a unique (in law) L\'evy process.
For a general cone $K$, any $K$-parameter L\'evy process $\{X_s\colon s\in K\}$
on $\R^d$ defined in
Pedersen and Sato \cite{PS04} induces a $K$-parameter convolution semigroup 
$\{\mu_s\}$ on $\R^d$ as $\mu_s=\law(X_s)$, the law of $X_s$.
But, for a given $K$-parameter convolution semigroup on $\R^d$, neither existence
nor uniqueness (in law) 
 of a $K$-parameter L\'evy process which induces the semigroup 
can be proved in general, as is shown in \cite{PS04}. 
The existence is proved when $d=1$, when
 $K$ is isomorphic to $\R_+^N$, or when $\mu_s$ does not
have Gaussian part for any $s$. The non-existence is proved for the canonical 
($d$-dimensional Gaussian) 
$\mbf S_d^+$-parameter convolution semigroup defined in \cite{PS04}
for $d\geqslant2$, where 
$\mbf S_d^+$ is the cone of $d\times d$ symmetric nonnegative-definite matrices.
Concerning the uniqueness, some sufficient conditions for
the uniqueness and for the non-uniqueness are given in \cite{PS04}.
For example, if $\{\mu_s\}$ is an $\R_+^2$-parameter convolution semigroup
on $\R$ such that the Gaussian part of $\mu_s$ is nonzero for any $s\neq0$,
then the corresponding $\R_+^2$-parameter L\'evy process on $\R$ is not
unique in law.
Subordination of a $K_2$-parameter L\'evy process on $\R^d$ by a 
$K_1$-parameter L\'evy process on $K_2$ results in a new $K_1$-parameter 
L\'evy process on $\R^d$, as is shown in Pedersen and Sato \cite{PS04} and earlier, 
in the case $K_2=\R_+^N$ and $K_1=\R_+$, in Barndorff-Nielsen, Pedersen, 
and Sato \cite{BPS01}. It induces subordination of a cone-parameter convolution 
semigroup.  But subordination of a cone-parameter convolution 
semigroup is not always accompanied by subordination of a cone-parameter L\'evy
process.

In this paper we give some results on inheritance of selfdecomposability, 
semi-selfdecomposability, and some related properties
 from subordinating to subordinated 
in subordination of cone-parameter convolution semigroups. 
Applications to distributions of type mult$G$ are given.

Semi-selfdecomposable distributions were introduced by Maejima and Naito 
\cite{MN98}. Their probabilistic representations were given by Maejima and
Sato \cite{MS03}. Their remarkable continuity properties were discovered by
Watanabe \cite{W00}. Recent papers of Kondo, Maejima, and Sato \cite{KMS06} and
Lindner and Sato \cite{LS07} studied
them in stationary distributions of some generalized Ornstein--Uhlenbeck processes.

%%%%%%%%%%%%%%%%%%%%%%%%%%%%
%%%%%%%%%%%%%%%%%%%%%%%%%%%%

\section{One-dimensional Gaussian subordinands}
Let $G_{a,\gm}$ denote Gaussian distribution on $\R$
with variance $a\geqslant0$ and mean $\gm\in\R$, where $G_{0,\gm}=
\dl_\gm$.
A $K$-parameter 
convolution semigroup $\{\mu_u\colon u\in K\}$ is called $1$-dimensional 
Gaussian if, for each $u\in K$, $\mu_u$ is $G_{a,\gm}$ with some $a$ and $\gm$.

A distribution $\mu$ on $\R^d$ is said to be selfdecomposable if, for each
$b>1$, there is a distribution $\mu'$ on $\R^d$  such that
\begin{equation}\label{2.1}
\wh\mu(z)=\wh\mu(b^{-1}z)\wh{\mu'}(z),\qquad z\in\R^d.
\end{equation}
Here $\wh\mu(z)$ and $\wh{\mu'}(z)$ are the characteristic functions of $\mu$
and $\mu'$, respectively.
If $\mu$ is selfdecomposable, then $\mu$ is infinitely divisible.

Noting that selfdecomposability is equivalent to semi-selfdecomposability 
with span $b$ for all $b>1$ (see Section 3 for the definition) and 
using Theorem 15.8 of \cite{S}, we see that an infinitely divisible
 distribution $\mu$ on $\R^d$ with L\'evy measure
$\nu$ is selfdecomposable if and only if 
\begin{equation}\label{2.1a}
\nu(b^{-1}B)\geqslant \nu(B)\qquad\text{for $b>1$ and $B\in\mcal B(\R^d
\setminus\{0\})$}.
\end{equation}
The condition \eqref{2.1a} holds if and only if 
$\nu$ has a polar representation
\begin{equation}\label{2.1b}
\nu(B)=\int_S \ld(d\xi)\int_0^{\infty} 1_B(r\xi)r^{-1}k_{\xi}(r)dr\qquad
\text{for }B\in\mcal B(\R^d\setminus\{0\}),
\end{equation}
where $S=\{\xi\colon |\xi|=1\}$, the unit sphere in $\R^d$, $\ld$ is a
measure on $S$, and $k_{\xi}(r)$ is a nonnegative function measurable in 
$\xi$ and decreasing in $r>0$ (Theorem 15.10 of \cite{S}). We are using
 the word {\it decrease} in the wide
sense allowing flatness.

\begin{thm}\label{thm1}
Let $K_1$ and $K_2$ be cones in\/ $\R^{N_1}$ and $\R^{N_2}$, respectively.  
Let $\{\mu_u\colon u\in K_2\}$ be a $1$-dimensional Gaussian $K_2$-parameter 
convolution semigroup (subordinand), $\{\rh_s\colon s\in K_1\}$ a
$K_1$-parameter 
convolution semigroup supported on $K_2$ (subordinating), and 
$\{\sg_s\colon s\in K_1\}$ the subordinated\/ $K_1$-parameter 
convolution semigroup on\/ $\R$. Fix $s\in K_1$.
If $\rh_s$ is selfdecomposable, then $\sg_s$ is selfdecomposable.
\end{thm}

We stress that
the Gaussian distribution $\mu_u$ is not necessarily
centered. For the centered Gaussian (that is strictly 2-stable), the 
result is largely extended in Theorem \ref{thm2} in Section 3.
Historically, Halgreen \cite{H79} raised a question equivalent 
to asking whether
the statement of Theorem \ref{thm1} for $K_1=K_2=\R_+$ is true.
After 22 years,
Theorem 1.1 of Sato \cite{S01} answered this question affirmatively.
The theorem above is an extension of it. 
In order to prove the theorem, we prepare a lemma.

\begin{lem}\label{lem1}
Let $f(r)$ be a nonnegative decreasing function of $r>0$ satisfying 
$\int_0^{\infty} (r\land1)r^{-1} f(r) dr<\infty$.  
Let $a\geqslant0$ and $\gm\in\R$.
Then, for every $b>1$ and $B\in\mcal B(\R\setminus\{0\})$,
\begin{equation}\label{2.2}
\int_0^{\infty} G_{ra,r\gm}(b^{-1}B)r^{-1}f(r)dr
\geqslant \int_0^{\infty} G_{ra,r\gm}(B)r^{-1}f(r)dr.
\end{equation}
\end{lem}

\begin{proof} Let $\{X_t\colon t\in\R_+\}$ be the L\'evy process with 
distribution $G_{a,\gm}$ at time $1$.  
Let $\{Z_t\colon t\in\R_+\}$ be a selfdecomposable subordinator with L\'evy
measure $r^{-1} f(r) dr$ and drift $0$. 
Let $\{Y_t\colon t\in\R_+\}$ be the 
L\'evy process on $\R$ obtained by subordination of $\{X_t\}$ by $\{Z_t\}$.
Then Theorem 30.1 of \cite{S} tells us that 
 the L\'evy measure $\nu^Y$ of $\{Y_t\}$ 
is expressed as 
\[
\nu^Y(B)=\int_0^{\infty} G_{ra,r\gm}(B)r^{-1}f(r)dr,\qquad 
B\in\mcal B(\R\setminus\{0\}).
\]
If $a>0$, then Theorem 1.1 of \cite{S01} establishes that 
$Y_t$ has a selfdecomposable distribution for any $t\geqslant0$. 
If $a=0$, then $\{X_t\}$ is a trivial L\'evy process (that is, $X_t=\gm t$, 
nonrandom) and $Y_t=\gm Z_t$, which has a selfdecomposable distribution. 
In any case, $\{Y_t\}$ is selfdecomposable. Hence $\nu^Y(b^{-1}B)\geqslant
\nu^Y(B)$, which is exactly \eqref{2.2}.
\end{proof}

\begin{proof}[Proof of Theorem \ref{thm1}] 
Let $\nu^{\mu_u}$, $\nu^{\rh_s}$, and $\nu^{\sg_s}$
denote the L\'evy measures of $\mu_u$, $\rh_s$, and $\sg_s$, respectively.
We have $\mu_u=G_{a_u,\gm_u}$ with some $a_u\geqslant0$ and $\gm_u\in\R$. 
These $a_u$ and $\gm_u$ are continuous functions of $u$ (Theorem 2.8 of 
\cite{PS03}). Since
$\mu_u$ has L\'evy measure $0$, Theorem 4.4 of \cite{PS03} says that
\[
\nu^{\sg_s}(B)=\int_{K_2}G_{a_u,\gm_u}(B)\nu^{\rh_s}(du),\qquad
B\in\mcal B(\R\setminus\{0\}).
\]

Assume that $\rh_s$ is selfdecomposable.  Then
$\nu^{\rh_s}$ is expressed as in the 
right-hand side of \eqref{2.1b} with $d=N_2$. Since $\mrm{Supp}(\rh_s)
\subseteq K_2$, it follows from Skorohod's
theorem \cite{Sk} (or Lemma 4.1 of \cite{PS03}) that the measure $\ld$
is supported on $S\cap K_2$ and that 
\[
\int_{S\cap K_2} \ld(d\xi)\int_0^{\infty} (r\land 1)r^{-1}k_{\xi}(r)dr
<\infty.
\]
For any $b>1$ and $B\in\mcal B(\R\setminus\{0\})$ we have
\begin{align*}
\nu^{\sg_s}(b^{-1}B)&=\int_{K_2}G_{a_u,\gm_u}(b^{-1}B)\nu^{\rh_s}(du)\\
&=\int_{S\cap K_2} \ld(d\xi)\int_0^{\infty} G_{a_{r\xi},\gm_{r\xi}}(b^{-1}B)
 r^{-1}k_{\xi}(r)dr=I\qquad\text{(say).}
\end{align*}
Notice that $k_{\xi}(r)$ is decreasing in $r$ and satisfies 
$\int_0^{\infty} (r\land1)r^{-1}k_{\xi}(r)dr<\infty$ for $\ld$-almost
every $\xi$ and that $a_{r\xi}=ra_{\xi}$ and $\gm_{r\xi}=r\gm_{\xi}$ 
(see Proposition 2.7 of \cite{PS03}).  
Thus we can apply Lemma \ref{lem1} to obtain
\[
I\geqslant \int_{S\cap K_2} \ld(d\xi)\int_0^{\infty} G_{a_{r\xi},\gm_{r\xi}}(B)
 r^{-1}k_{\xi}(r)dr=\nu^{\rh_s}(B).
\]
This means that $\sg_s$ is selfdecomposable.
\end{proof}

\begin{rem}\label{rm0}
Let $K$ be a cone and let $\{\mu_s\colon s\in K\}$ be a $K$-parameter
convolution semigroup on $\R^d$. Let $s_0\in K\setminus\{0\}$.  If
$\mu_{s_0}$ is selfdecomposable, then $\mu_{ts_0}$ equals selfdecomposable
for all $t\geqslant0$ since $\mu_{ts_0}=\mu_{s_0}^t$, the $t$\,th convolution power of
$\mu_{s_0}$ (Proposition 2.7 of \cite{PS03}), 
but $\mu_{s_1}$ may not be selfdecomposable for some
$s_1\in K\setminus\{ts_0\colon t\geqslant0\}$. This follows from Sections 
2 and 3 of \cite{PS03}.
\end{rem}

\begin{rem}\label{rm1}
In Theorem \ref{thm1} let $K_1=K_2=\R_+$ and replace \lq\lq Gaussian" 
by \lq\lq $\al$-stable (not necessarily strictly $\al$-stable)", where 
$\al\in(0,2]$. Then the statement for $\al=2$ is exactly Theorem 1.1 of
\cite{S01}. The statement for $\al\in(1,2)$ is not true, which is pointed out
by Kozubowski \cite{K05} using Theorem 2.1(v) of Ramachandran \cite{R97}.
It is not known whether the statement for $\al\in(0,1]$ is true.
\end{rem}

\begin{rem}\label{rm2}
If $\mu$ is selfdecomposable, then 
the distribution $\mu'$ in \eqref{2.1} is uniquely
determined by $\mu$ and $b$, and $\mu'$ is also infinitely divisible.
For nonnegative integers $m$ we define $L_m(\R^d)$ as follows: $L_0(\R^d)$ 
is the class of selfdecomposable distributions on $\R^d$; 
for $m\geqslant1$, $L_m(\R^d)$ is the class of $\mu\in L_0(\R^d)$ such that, 
for every $b>1$, $\mu'$ in \eqref{2.1} belongs to $L_{m-1}(\R^d)$. Thus
we get a strictly decreasing sequence of subclasses of the class $ID(\R^d)$ of
infinitely divisible distributions on $\R^d$.
We define $L_{\infty}(\R^d)$ as
the intersection of $L_m(\R^d)$, $m=0,1,2,\dots$ . 
It is not known even in the case $K_1=K_2=\R_+$ whether Theorem \ref{thm1} 
is true with \lq\lq selfdecomposable" replaced by \lq\lq of class $L_m$" 
for $m\in\{1,2,\ldots,\infty\}$.
\end{rem}

\begin{rem}\label{rm3}
Let $d\geqslant2$.
Theorem \ref{thm1} cannot be generalized to $d$-dimensional Gaussian.
If $\{\mu_u\colon u\in \R_+\}$ is an $\R_+$-parameter convolution semigroup
(subordinand) induced by $d$-dimensional Brownian motion with nonzero drift
and $\{\rh_t\colon t\in \R_+\}$ is an $\R_+$-parameter 
convolution semigroup supported on $\R_+$ 
(subordinating) of Thorin class (of generalized gamma convolutions,
in other words) satisfying some additional condition, 
then the subordinated $\R_+$-parameter 
convolution semigroup $\{\sg_t\colon t\in \R_+\}$ on $\R^d$
is not selfdecomposable for any $t>0$.
This fact was noticed by Takano \cite{T89}
and Grigelionis \cite{G07a}.
Recall that the Thorin class is a subclass of the class of 
selfdecomposable distributions.
This $\sg_t$ supplies an example of an infinitely divisible
 non-selfdecomposable distribution
whose one-dimensional projections are selfdecomposable, since we can apply
Theorem 1.1 of \cite{S01} to one-dimensional projections of 
$\{\mu_u\colon u\in \R_+\}$. 
The first example  of a distribution with this projection property was constructed 
in Sato \cite{S98}.
\end{rem}

\begin{rem}\label{rm4}
It is not known even in the case $K_1=K_2=\R_+$ whether Theorem \ref{thm1} 
is true with \lq\lq selfdecomposable" replaced by \lq\lq semi-selfdecomposable",
which will be defined in the next section.
\end{rem}

%%%%%%%%%%%%%%%%%%%%%%%%%%%%%%
%%%%%%%%%%%%%%%%%%%%%%%%%%%%%%
\section{Inheritance of semi-selfdecomposability}

A distribution on $\R^d$ is called semi-selfdecomposable if there are
$b>1$ and $\mu'\in ID(\R^d)$ such that
\begin{equation}\label{3.1}
\wh\mu(z)=\wh\mu(b^{-1}z)\wh{\mu'}(z),\qquad z\in \R^d.
\end{equation}
The $b$ in this definition is called a span of $\mu$; it is not uniquely
determined by $\mu$.
The class of semi-selfdecomposable distributions on $\R^d$ having $b$ as a
span is denoted by $L_0(b^{-1},\R^d)$.
If $\mu\in L_0(b^{-1},\R^d)$, then $\mu$ is infinitely divisible and 
the distribution $\mu'$ is uniquely determined by $\mu$ and $b$.
For any positive integer $m$ we inductively define
\[
L_m(b^{-1},\R^d)=\{\mu\in L_0(b^{-1},\R^d)\colon \mu'\in L_{m-1}
(b^{-1},\R^d)\}.
\]
Then $L_{m}(b^{-1},\R^d)$ is a subclass of $L_{m-1}(b^{-1},\R^d)$. 
In fact we can prove that the former is a strict subclass of the latter 
(see Remark 3.1 of \cite{MSW99}). 
Further we define $L_{\infty}(b^{-1},\R^d)$ as the intersection of
$L_m(b^{-1},\R^d)$ for $m=0,1,\ldots$.

Let $0<\al\leqslant2$.  A distribution $\mu$ on $\R^d$ is called
strictly $\al$-semistable if $\mu\in ID(\R^d)$ and if there is a real
number $b>1$ such that
\begin{equation}\label{3.2}
\wh\mu(z)^{b^{\al}}=\wh\mu(b z),\qquad z\in\R^d,
\end{equation}
or, equivalently, $\wh\mu(z)^{b^{-\al}}=\wh\mu(b^{-1} z)$, $z\in\R^d$.
In this case we say that the $\al$-semistable distribution $\mu$ has
a span $b$, which is not uniquely determined by $\mu$.
If $\mu$ is strictly $\al$-semistable on $\R^d$ with a span $b$, 
then it is easy to see that $\mu\in L_{\infty}(b^{-1},\R^d)$,
since we have
\[
\wh\mu(z)=\wh\mu(z)^{b^{-\al}}\wh\mu(z)^{1-b^{-\al}}=\wh\mu(b^{-1}z)
\wh\mu(z)^{1-b^{-\al}}.
\]

For description and examples of L\'evy measures of semi-selfdecomposable
and semistable distributions, see Sections 14 and 15 of \cite{S}.

The statement of Remark \ref{rm0} is true also for 
\lq\lq semi-selfdecomposable with a span $b$" and \lq\lq strictly 
$\al$-semistable with a span $b$" in place of \lq\lq selfdecomposable".

\begin{thm}\label{thm2}
Let $K_1$ and $K_2$ be cones in\/ $\R^{N_1}$ and $\R^{N_2}$, respectively.  
Let $\{\mu_u\colon u\in K_2\}$ be a $K_2$-parameter 
convolution semigroup on $\R^d$ (subordinand), $\{\rh_s\colon s\in K_1\}$ a
$K_1$-parameter 
convolution semigroup supported on $K_2$ (subordinating), and 
$\{\sg_s\colon s\in K_1\}$ the subordinated $K_1$-parameter 
convolution semigroup on $\R^d$.
Suppose that there are $0<\al\leqslant2$ and $b>1$ such that, for every $u\in K_2$, 
$\mu_u$ is strictly $\al$-semistable with a span $b^{1/\al}$.
Fix $s\in K_1$. Then the following statements are true.

(i) Let $m\in\{0,1,\ldots,\infty\}$.  If
\begin{equation}\label{3.3}
\rh_s\in L_m(b^{-1}, \R^{N_2}),
\end{equation}
then
\begin{equation}\label{3.4}
\sg_s\in L_m(b^{-1/\al}, \R^d).
\end{equation}

(ii) Let $0<\al'\leqslant1$.  If
\begin{equation}\label{3.5}
\text{$\rh_s$ is strictly $\al'$-semistable with a span $b$},
\end{equation}
then 
\begin{equation}\label{3.6}
\text{$\sg_s$ is strictly $\al\al'$-semistable with a span $b^{1/\al}$}.
\end{equation}
\end{thm}

Note that strictly 1-semistable distributions supported on a cone are
delta distributions.
This theorem is an extension of Theorem 4.10 of Pedersen and Sato \cite{PS03}
to the \lq\lq semi" case.
We prepare a lemma. This is an analogue of Lemma 4.11 of 
\cite{PS03} and the proof is almost the same.

\begin{lem}\label{lem2}
Let $K_2$ be a cone in $\R^{N_2}$. Suppose that $\rh$ is in $L_0(b^{-1},
\R^{N_2})$ 
and that $\mrm{Supp}(\rh)\subseteq K_2$. Let $\rh'$ be defined by
$\wh\rh(z)=\wh\rh(b^{-1}z)\wh{\rh'}(z)$, $z\in\R^{N_2}$.
Then $\mrm{Supp}(\rh')\subseteq K_2$.
\end{lem}

\begin{proof}[Proof of Theorem \ref{thm2}]
Let us prove assertion (i) for $m=0$. Assume that $\rh_s\in L_0(b^{-1},\R^{N_2})$.
Define $\rh''_s$ as $\wh{\rh''_s}(z)=\wh\rh_s(b^{-1}z)$.
Then 
\[
\wh\rh_s(z)=\wh{\rh''_s}(z) \wh{\rh'_s}(z)
\]
and thus $\rh_s=\rh''_s * \rh'_s$. 
Lemma \ref{lem2} tells us that $\rh'_s$ is supported on $K_2$. 
Clearly $\rh''_s$ is also supported on $K_2$.
Hence
{\allowdisplaybreaks
\begin{align*}
\wh\sg_s(z)&=\int_{K_2}\wh\mu_u(z) \rh_s(du)=\iint_{K_2\times K_2}
\wh\mu_{u_1+u_2}(z)\rh''_s(du_1)\rh'_s(du_2)\\
&=\iint_{K_2\times K_2}
\wh\mu_{u_1}(z)\wh\mu_{u_2}(z)\rh''_s(du_1)\rh'_s(du_2)\\
&=\int_{K_2}\wh\mu_{b^{-1}u_1}(z)\rh_s(du_1)
\int_{K_2}\wh\mu_{u_2}(z)\rh'_s(du_2).
\end{align*}
Using Proposition 2.7 of \cite{PS03} and the assumption that} 
$\mu_u$ is strictly $\al$-semistable with a span $b^{1/\al}$,
we have
\[
\wh\mu_{b^{-1}u}(z)=\wh\mu_u(z)^{b^{-1}} =\wh\mu_u(b^{-1/\al}z).
\]
It follows that
\begin{equation}\label{3.7}
\wh\sg_s(z)=\wh\sg_s(b^{-1/\al}z)\int_{K_2}\wh\mu_{u}(z)\rh'_s(du).
\end{equation}
Since $\int_{K_2}\wh\mu_{u}(z)(\rh'_s)^t(du)$ is subordination of 
$\{\mu_u\}$ by $\{(\rh'_s)^t\colon t\in\R_+\}$, we see that 
$\int_{K_2}\wh\mu_{u}(z)\rh'_s(du)$ is
infinitely divisible. This shows that $\sg_s\in L_0(b^{-1/\al},\R^d)$.

Next, we assume that (i) is true for a fixed $m\in\{0,1,\ldots\}$.
We claim that (i) is true for $m+1$.
Suppose that $\rh_s\in L_{m+1}(b^{-1}, \R^{N_2})$. Then
$\wh\rh_s(z)=\wh\rh_s(b^{-1}z) \wh{\rh'_s}(z)$ with 
$\rh'_s\in L_{m}(b^{-1}, \R^{N_2})$.
We have \eqref{3.7} since $L_{m+1}(b^{-1}, \R^{N_2})\subseteq
L_0(b^{-1}, \R^{N_2})$. Now 
$\int_{K_2}\wh\mu_{u}(z)(\rh'_s)^t(du)$ is subordination such that
$(\rh'_s)^t$ is in $L_m(b^{-1}, \R^{N_2})$. Hence 
$\int_{K_2}\wh\mu_{u}(z)\rh'_s(du)$ is the 
the characteristic function of a distribution in
 $L_m(b^{-1/\al},\R^d)$. It follows that
 $\sg_s\in L_{m+1}(b^{-1/\al}, \R^d)$,
which shows (i) for $m+1$.

Assertion (i) for $m=\infty$ is a consequence of that for finite $m$.

To prove (ii), assume \eqref{3.5}. Let us show \eqref{3.6}, that is,
\begin{equation}\label{3.8}
\wh\sg_s(z)^{b^{\al'}}=\wh\sg_s(b^{1/\al}z).
\end{equation}
Using
{\allowdisplaybreaks
\begin{align*}
\wh\rh_{b^{\al'}s}(z)&=\wh\rh_s(z)^{b^{\al'}}=\wh\rh_s(bz)\\
\intertext{and}
\wh\mu_{bu}(z)&=\wh\mu_u(z)^{b}=\wh\mu_u(b^{1/\al}z),
\end{align*}
we obtain
\begin{align*}
\wh\sg_s(z)^{b^{\al'}}&=\wh\sg_{b^{\al'}s}(z)=\int_{K_2}\wh\mu_u(z)
\rh_{b^{\al'}s}(du)=\int_{K_2}\wh\mu_{bu}(z)\rh_{s}(du)\\
&=\int_{K_2}\wh\mu_u(b^{1/\al}z)\rh_s(du)=\wh\sg_s(b^{1/\al}z),
\end{align*}
completing the proof.}
\end{proof}

{\it Application to distributions of type $\mrm{mult}G$}. 
Following Barndorff-Nielsen and P\'erez-Abreu \cite{BP03}, we say that
a probability measure $\sg$ on $\R^d$ 
is of type mult$G$ if $\sg=\law(Z^{1/2}X)$, 
where $X$ is a standard Gaussian on $\R^d$, $Z$ is an 
$\mbf S_d^+$-valued infinitely divisible random variable, $Z^{1/2}$
is the nonnegative-definite symmetric square root of $Z$, and $X$
and $Z$ are independent. Here, as in Section 1, $\mbf S_d^+$ is the
class of of $d\times d$ symmetric nonnegative-definite matrices and elements
of $\R^d$ are considered as column $d$-vectors.
Regarding the lower triangle $(s_{jk})_{k\leqslant j}$ of
$s=(s_{jk})_{j,k=1}^d \in \mbf S_d^+$ as a $d(d+1)/2$-vector, $\mbf S_d^+$
is identified with a cone in $\R^{d(d+1)/2}$. The $\mbf S_d^+$-parameter
convolution semigroup $\{\mu_s\colon s\in\mbf S_d^+\}$ on $\R^d$ where 
$\mu_s$ is $d$-dimensional Gaussian with mean vector $0$ and covariance
matrix $s$ is called the canonical $\mbf S_d^+$-parameter convolution 
semigroup (\cite{PS03}). The following fact is known (Theorem 4.7 of 
\cite{PS03} and its proof).

\begin{prop}\label{prop1}
Let $\{\mu_u\colon u\in \mbf S_d^+\}$ be 
the canonical\/ $\mbf S_d^+$-parameter convolution semigroup (subordinand),
$\{\rh_t\colon t\in\R_+\}$ an 
$\R_+$-parameter 
convolution semigroup on $\R^{d(d+1)/2}$
supported on $\mbf S_d^+$ (subordinating), and $\{\sg_t\colon t\in\R_+\}$ 
the subordinated $\R_+$-parameter convolution
semigroup on\/ $\R^d$. Then $\sg_1$ (or, more generally, $\sg_t$) 
 is of type\/ {\rm mult}$G$. 
  Conversely, any distribution on $\R^d$ of type\/ {\rm mult}$G$ 
is expressible as $\sigma_1$ of such an $\R_+$-parameter convolution
semigroup $\{\sigma_t\colon t\in\R_+\}$.
The correspondence of the two representations of a distribution of
type\/ {\rm mult}$G$ is that $\rh_1=\law(Z)$.
\end{prop}

We can show the following.

\begin{prop}\label{prop2}
Let $\sg$ be a distribution of type\/ {\rm mult}$G$, that is, let 
$\sg=\law(Z^{1/2}X)$, where $X$ is a standard 
Gaussian on $\R^d$, $Z^{1/2}$
is the nonnegative-definite symmetric square root of\/
$\mbf S_d^+$-valued infinitely divisible random variable  $Z$, and $X$
and $Z$ are independent. 

(i) Let $m\in\{0,1,\ldots,\infty\}$ and $b>1$.  If
$\law (Z)\in L_m(b^{-1}, \R^{d(d+1)/2})$, then 
$\sg\in L_m(b^{-1/2}, \R^d)$.

(ii) Let $0<\al'\leqslant1$ and $b>1$.  If $\law(Z)$
is strictly $\al'$-semistable with a span $b$, then 
$\sg$ is strictly $2\al'$-semistable with a span $b^{1/2}$.
\end{prop}

\begin{proof} Recall that a distribution $\mu$ is strictly $\al$-stable
if and only if it is strictly $\al$-semistable with a span $b$ for all
$b>1$. Apply Theorem \ref{thm2} combined with Proposition \ref{prop1}.
\end{proof}

\end{document}